\documentclass[preprint,authoryear,12pt]{elsarticle}

\usepackage{amssymb}
\usepackage{amsthm}
\usepackage{amsfonts}

\newtheorem{theorem}{Theorem}[subsection]
\newtheorem{lemma}[theorem]{Lemma}

\newtheorem{corollary}[theorem]{Corollary}

\journal{arXiv}

\begin{document}

\begin{frontmatter}



\title{About Cantor Works on Trigonometric Series}


\author[ad1,sbu]{Muhammad-Ali A'rabi}
\author[ad3,sbu]{Farnaz Irani}

\address[ad1]{m\_aerabi@mehr.sharif.ir}
\address[ad3]{f.irani@modares.ac.ir}
\address[sbu]{Shahid Beheshti University \\ Daneshjoo st., Velenjak, Tehran, I.R.Iran}

\begin{abstract}
This paper is an investigation into Cantor works about representing a function with trigonometric series, and his proofs about its uniqueness. These works are important, because they cause invention of point-set topology, and foundation of basic ideas that led Cantor to his set theory.

\end{abstract}


\end{frontmatter}


\section{Introduction}
\label{intro}

The essay on representing functions with trigonometric series written by Riemann, ``\"Uber die Darstellung einer Funktion durch eine trigonometrische Reihe" (1854), was published after his death in 1867 by Dedekind. Cantor was influenced in this work, and started to researching uniqueness of this representation. Cantor's essays in this area are:
\begin{itemize}
\item ``\"Uber einen die trigonometrischen Reihen betreffenden Lehrsatz" (1870)
\item ``Beweis, da\ss\ eine f\"ur jeden reellen Wert von $x$ durch eine trigonometrische Reihe gegebene Funktion $f(x)$ sich nur auf eine einzige Weise in dieser Form darstellen l\"a\ss t" (1870)
\item ``Notiz zu dem vorangehenden Aufsatze" (1871)
\item ``\"Uber trigonometrische Reihen" (1871)
\item ``\"Uber die Ausdehung eines Satzes aus der Theorie der trigonometrische Reihen" (1872)
\item ``Bemerkung \"uber trigonometrische Reihen" (1880)
\item ``Fernere Bemerkung \"uber trigonometrische Reihen" (1880)
\end{itemize}
He proved in his 1870a work that if $f(x)$ is represented by a trigonometric series convergent for all $x$, then the representation is unique. In his 1871 works he strengthened the result, proving that the uniqueness holds even if the series diverges at a finite number of points in any given interval. In his 1872 work he strengthened his result, even further. In this paper, will investigate in Cantor works of these 3 years.

\section{On Uniqueness of Representation}
\label{main}

In this paper, a non-negative integer is called a \emph{whole} integer. A positive real number is called a \emph{size}.

\subsection{\"Uber einen die trigonimetrischen Reihen betreffenden Lehrsatz}
\label{c1870a}

This essay consists of 4 sections of discusion, preceded two sections of conclusions, ``Erster Fall" and ``Zweiter Fall".
\par
\begin{lemma}
Let $\{n_k\}_{k \in \mathbb{N}}$ be an infinite sequence of whole positive integers, with property that $n_k > 2^k n_{k-1}$. Then, exists a size $\Omega$ with the property that for any $n \in \{n_k\}_{k \in \mathbb{N}}$,
$$n\Omega = 2 z_n + 1 \pm \Theta_n,$$
in which $z_n$ is some whole positive integer and $\lim_{n \rightarrow \infty}{\Theta_n} = 0$.
\end{lemma}
\par
\begin{proof}
We build a new sequence of odd whole integers $\{2m_k + 1\}_{k \in \mathbb{N}}$ based on $\{n_k\}_{k \in \mathbb{N}}$ with this rule:
$$\left | 2m_k + 1 - \frac{n_{k+1}}{n_k} \right | \leqq 1.$$
Obviously, there is exactly one integer that satisfies condition above. This property concludes that
$$\left | 2m_k + 1 - (2m_{k-1} + 1) \frac{n_{k+1}}{n_k} \right | \leqq 1,$$
for all $k \in \mathbb{N} \setminus \{1\}$. Now we construct a newer sequence $\{q_k\}_{k \in \mathbb{N}}$ of rational numbers $q_k$ in which
$$ q_1 = \frac{1}{n_1}, q_k = \frac{2m_{k-1} + 1}{n_k}.$$
Note that sequence $\{q_k\}_{k \in \mathbb{N}}$ approaches to a non-negative value as $k \rightarrow \infty$, because
$$\left | q_{k-1} - q_k \right | = \left | \frac{2m_{k-2} + 1}{n_{k-1}} - \frac{2m_{k-1} + 1}{n_k} \right | \leqq \frac{1}{n_k}.$$
Note that difference between $q_s$ and any other member $q_t \in \{q_k\}_{k \in \mathbb{N}}$ such that $t>s$, is not greater than 
$$\sum_{k=s}^{\infty} \frac{1}{n_{k+1}}.$$
This inequality guarantees convergence of the sequence. Let $\Omega = \lim_{k \rightarrow \infty}{q_k}$. So we have
$$\left | \Omega - q_s \right | = \left | \Omega - \frac{2m_{s-1}+1}{n_s} \right | \leqq \sum_{k=s+1}^{\infty} \frac{1}{n_{k+1}},$$
hence
$$\left | n_s \Omega - (2m_{s-1}+1) \right | \leqq n_s \sum_{k=s+1}^{\infty} \frac{1}{n_{k+1}} \lneqq \frac{1}{2^s}.$$
Let $z_{n_s} = m_{s-1}$, so we have 
$$n\Omega = 2 z_n + 1 \pm \Theta_n,$$
with $\Theta_{n_s} < 2^{-s}$, so $\lim_{n \rightarrow \infty}\Theta_n = 0$.
\end{proof}

\par
\begin{lemma}
We can modify the proof represented here, so that $\Omega$ lie in any given interval with real boundaries.
\end{lemma}
\begin{proof}
If we divide interval $[0,2]$ into $2\nu$ equal intervals, supose that we want $\Omega$ to lie in $\mu$th of them. Let $\nu_1$ be the first member of $\{n_k\}_{k \in \mathbb{N}}$ satisfying the condition $\nu_1 > 6\nu$, so exists one specific odd integer $2\mu_1 + 1$ so that
$$\frac{2\mu_1 + 1}{\nu_1} \in \left [ \frac{3\mu - 2}{3\nu}, \frac{3\mu - 1}{3\nu} \right ].$$
We can find integers $\nu_k \in \{n_k\}_{k \in \mathbb{N}}$ and $\mu_k$ the same way, so that
$$\left | 2\mu_k + 1 - (2\mu_{k-1} + 1)\frac{\nu_{k}}{\nu_{k-1}} \right | \leqq 1,$$
hence
$$\left | \frac{2\mu_{k-1} + 1}{\nu_{k-1}} - \frac{2\mu_{k} + 1}{\nu_{k}} \right | \leqq \frac{1}{\nu_k}.$$
Let $\xi_k = (2\mu_k + 1)/\nu_k$, and $\Omega = \lim_{k \rightarrow \infty} \xi_k$. Now with $z_{\nu_k} = \mu_k$, last proof is still correct, and $\Omega$ lies in $[(\mu-1)/\nu, \mu/\nu]$.
\end{proof}

\par
\begin{lemma}
Let $\{\varrho_n\}_{n \in \mathbb{N}}$ be a sequence of sizes, with property that in every subsequence $\{\varrho_{n_k}\}_{k \in \mathbb{N}}$ of it, exists one member lesser than arbitrary size $\delta$. Then $\lim_{n \rightarrow \infty} \varrho_n = 0$.
\end{lemma}
\begin{proof}
Let $\{\Delta_n\}_{n \in \mathbb{N}}$ be a strictly decreasing sequence of sizes converging zero, for example $\{\frac{1}{n}\}_{n \in \mathbb{N}}$. Delete from the $\{\varrho_n\}_{n \in \mathbb{N}}$ members bigger than $\Delta_1$, then delete members bigger than $\Delta_2$, and continue this act. With none of these operations we deleted infinitely many members, otherwise, deleted members form a subsequence of $\{\varrho_n\}_{n \in \mathbb{N}}$ with no members smaller than some $\Delta_k$. So we have a last deleted member $\varrho_{n_k}$ in deleting operation for $\Delta_k$, in which every member after that is smaller than $\Delta_k$. So $\lim_{n \rightarrow \infty} \varrho_n = 0$.
\end{proof}

\par
\begin{corollary}
Let $\{\varrho_n\}_{n \in \mathbb{N}}$ be a sequence of sizes, with property that every subsequence $\{\varrho_{n_k}\}_{k \in \mathbb{N}}$ of it, itslef has a subsequence converging zero. Then $\lim_{n \rightarrow \infty} \varrho_n = 0$.
\end{corollary}

\par
\begin{theorem}
\emph{(Lehrsatz)} When for each real value of $x$ between given boundaries $(a<x<b)$,
$$\lim_{n \rightarrow \infty} (a_n \sin{nx}+b_n \cos{nx})=0,$$
then
$$\lim_{n \rightarrow \infty} a_n=0, \mathrm{\ and\ } \lim_{n \rightarrow \infty} b_n=0.$$
\end{theorem}
\begin{proof}
We will transform $a_n \sin{nx} + b_n \cos{nx}$ into form $\varrho_n cos{(\varphi_n - nx)}$, in which $\varrho_n = \sqrt{a^2_n + b^2_n}$, and $\varphi_n \in [0, 2\pi]$.
\end{proof}

\section{Notes on Cantor Works}
\label{notes}

\subsection{Language}
\label{lang}
Language Cantor used in his works is different with nowaday scientific writing ethics. First difference, of course, is because he writes in German, but it is not all. Away from order in sentences, he sometimes gives details about things obvious to us, and sometimes leaves an amibigous thing without any further comments. It is also a little complicated to follow things. Sometimes you won't know the proposition he mentioned is proved, or he will give a proof later.
\par
To comment on some of his/German terminology, let us begin with ``Reihe". This word means both ``sequence" and ``series", for example in ``trigonometri-sche Reihe" and ``Gr\"o\ss enreihe". ``Gr\"o\ss e" itself, literally ``size", means a non-negative real value. ``Zahl", literally ``number" (as in ``Zahlengr\"o\ss e" that is equivalent to ``Gr\"o\ss e"), means ``integer". As adjectives, ``reellen" for is used for indicating a number a real (e.g. ``reellen Gr\"o\ss engebietes"), ``ganzen" to indicate a number is non-negative (e.g. ``positiven ganzen Zahlen"), etc. The adjective ``ganzen", literally ``whole", sometimes used in English. Example is the set of whole numbers $\mathbb{W} = \mathbb{N} \cup \{0\}$.

\subsection{Notation}
\label{notation}
Notation used by Cantor is also so different with standard notation we use in our scientific literature. Actually, notation we use now was introduced in Hilbert's ``Grundlagen der Mathematik" in 1930s.
\par
As an example, Cantor had prefered $a, b, c, ...$, $u, v, w, x, ...$, and most important $a, a', a'', ..., a^{(n)}, ...$, instead of using indices. As one reason maybe publishers prefered these ways because these methodes were easier to typeset, and them became popular through mathematical publications. By the way, Cantor had prefered to use $a^{(n)}$ instead of $a_n$, when he had to address $n$th term of a sequence. So, just the same, in his work of 1882, he designate set of limit points of a set $A$ with $A'$, to construct his sequence of sets. The reason set of limit points of a set $A$ is designated by $A'$ in nowaday analysis standard notation, is this. Maybe if Cantor had used set of boundary points instead for his proof, we were now using $A'$ instead of $\partial{A}$.

\bibliographystyle{elsarticle-harv}



\end{document}